\documentclass[11pt,leqno]{article}
\usepackage{amsmath,amssymb,amsthm,latexsym}
\usepackage{indentfirst}
\usepackage{amsmath}

\newtheorem{theorem}{\bf \large Theorem}[section]
\newtheorem{PROPOSITION}{\bf \large Proposition}[section]
\newtheorem{corollary}{\bf \large Corollary}[section]
\newtheorem{lemma}{\bf \large Lemma}[section]
\newtheorem{ex}{\bf \large Example}[section]

\title{\textbf{Generic conformally flat hypersurfaces  in $\mathbb{R}^4$}}
\author {\normalsize {Xiu Ji, ~~ TongZhu Li\thanks{The work  is  supported by the grant No.
11571037 and No.11471021 of NSFC.}}\\
\small{Department of Mathematics, Beijing institute of
technology,}\\
\small{Beijing,100081,China.}\\
\small{ E-mail:jixiu1106@163.com,~~litz@bit.edu.cn.}}

\date{}

\begin{document}
\maketitle
\begin{abstract}
In this paper, we study generic conformally flat hypersurfaces in the Euclidean $4$-space $\mathbb{R}^4$
using the framework of M\"{o}bius geometry.
First, we classify  locally the generic conformally flat  hypersurfaces with closed M\"{o}bius form
 under the M\"{o}bius transformation group of $\mathbb{R}^4$. Such examples come from cones, cylinders,
or rotational hypersurfaces over the  surfaces with constant Gaussian curvature in $3$-spheres, Euclidean
$3$-spaces, or hyperbolic $3$-spaces, respectively. Second, we investigate the global behavior of the generic conformally flat hypersurface
 and give some integral
formulas about these hypersurfaces.
\end{abstract}

\vskip 1 cm
\section{Introduction}
A Riemannian manifold $(M^n, g)$ is conformally
flat, if every point has a neighborhood which is conformal to an open
set in the Euclidean space $\mathbb{R}^n$. A hypersurface of the Euclidean space $\mathbb{R}^{n+1}$
 is said to be conformally flat if so it is with respect to the induced
metric. The dimension of the hypersurface seems to play an important role
in the study of conformally flat hypersurfaces. For $n\geq4$, the immersed hypersurface  $f: M^n \to \mathbb{R}^{n+1}$ is conformally flat
if and only if at least $n-1$ of the principal curvatures coincide
at each point by the result of Cartan-Schouten
(\cite{cartan},\cite{schou}). Cartan-Schouten's result is no longer true for three dimensional hypersurfaces. Lancaster (\cite{lan}) gave some examples of conformally flat
hypersurfaces in $\mathbb{R}^4$ having three different principal curvatures. For $n=2$,
the existence of isothermal coordinates means that any Riemannian
surface is conformally flat.

 A conformally flat hypersurface $f:M^3\to \mathbb{R}^4$ in $\mathbb{R}^4$
is said to be generic, if the second fundamental form has three distinct
eigenvalues everywhere on $M^3$. Standard example of generic conformally flat hypersurface comes from cone, cylinder,
or rotational hypersurface over a  surface with constant Gaussian curvature in $3$-sphere $\mathbb{S}^3$, Euclidean
$3$-space $\mathbb{R}^3$, or hyperbolic $3$-space $\mathbb{H}^3$, respectively. The (local) classification
of these hypersurfaces is far from complete. However, several partial classification results
of generic conformally
flat hypersurfaces were given in \cite{je1}, \cite{je3}, \cite{je4},\cite{je5}, \cite{la},\cite{su1},\cite{su2}
and \cite{su3}.

It is known that the conformal transformation group of $\mathbb{R}^n$ is isomorphic to
its M\"{o}bius transformation group if $n\geq 3$.  As conformal invariant objects, generic conformally flat hypersurfaces are investigated
in this paper using the framework of M\"{o}bius geometry. If  an immersed hypersurface in $\mathbb{R}^{n+1}$
has not any umbilical point, then we can define the so-called M\"{o}bius metric on the hypersurface, which
is  invariant under M\"{o}bius transformations \cite{w}.
Together with another quadratic form (called the M\"{o}bius second fundamental form)
they form a complete system of invariants for hypersurfaces $(dim\geq3)$ in M\"obius geometry \cite{w}.
Other important M\"{o}bius invariants of the hypersurface are the M\"{o}bius form and the Blaschke tensor.
First, we find that the standard examples of generic conformally flat hypersurface has closed  M\"{o}bius form, and vice versa.
\begin{theorem}\label{the1}
Let $f:M^3\rightarrow \mathbb{R}^{4}$ be a generic conformally flat  hypersurface.
The M\"{o}bius form is closed if and only if
the hypersurface $f$  is locally M\"{o}bius equivalent to one of the following hypersurfaces in $\mathbb{R}^{4}$:\\
$(1)$ a cylinder over a surface in $\mathbb{R}^3$ with constant Gaussian curvature,\\
$(2)$ a cone over a surface in $\mathbb{S}^3$ with constant Gaussian curvature, \\
$(3)$ a rotational hypersurface over a  surface in $\mathbb{H}^3$ with constant Gaussian curvature.
\end{theorem}
Second, we investigate the global behavior of compact generic conformally flat hypersurfaces
by the M\"{o}bius invariants. Let $(M^n,g)$ be a Riemannian manifold. $K(P)$ denotes the sectional curvature of sectional plane $P (\in \wedge^2TM^n)$.
We call the sectional curvature $K(P)$ have sign if $K(P)\geq 0$ for all $P\in \wedge^2TM^n$, or $K(P)\leq 0$ for all $P\in \wedge^2TM^n$.
\begin{theorem}\label{the2}
Let $f:M^3\rightarrow \mathbb{R}^{4}$ be a generic conformally flat  hypersurface.
If the hypersurface $M^3$ is compact,
then the sectional curvature of the M\"{o}bius metric can not have sign.
\end{theorem}
\begin{theorem}\label{the3}
Let $f:M^3\rightarrow \mathbb{R}^{4}$ be a generic conformally flat  hypersurface.
If the hypersurface $M^3$ is compact, then
$$\int_{M^3}\{|\tilde{A}|^2+\frac{1}{3}R^2-|Ric|^2-\frac{2}{27}\}dv_g=0,$$
where  $\tilde{A}:=A-\frac{1}{3}tr(A)g$ denotes the trace-free Blaschke tensor, $|Ric|$ denotes the norm of the Ricci curvature of $g$,
and $R$ denotes the scalar curvature of $g$.
\end{theorem}
\begin{corollary}\label{cor}
Let $f:M^3\rightarrow \mathbb{R}^{4}$ be a generic conformally flat  hypersurface.
If the hypersurface $M^3$ is compact, then
$$\int_{M^3}\{|\tilde{A}|^2-\frac{2}{27}\}dv_g>0,$$
where  $\tilde{A}:=A-\frac{1}{3}tr(A)g$ denotes the trace-free Blaschke tensor.
\end{corollary}

The paper is organized as follows. In section 2, we review the
elementary facts about M\"{o}bius geometry of hypersurfaces in
$\mathbb{R}^{n+1}$. In section 3, we investigate local behavior of generic conformally flat hypersurfaces in $\mathbb{R}^4$ and
prove Theorem \ref{the1}. In section 4, we investigate global behavior  of generic conformally flat hypersurfaces in $\mathbb{R}^4$ and
prove Theorem \ref{the2} and Theorem \ref{the3}.

\vskip 1 cm
\section{M\"{o}bius invariants of hypersurfaces in $\mathbb{R}^{n+1}$}
In \cite{w}, Wang has defined M\"{o}bius invariants of submanifolds in $\mathbb{S}^{n+1}$ and given
a congruent theorem
of hypersurfaces in $\mathbb{S}^{n+1}$. In this section,
we define M\"{o}bius invariants and give a congruent theorem
of hypersurfaces in $\mathbb{R}^{n+1}$ in the same way in \cite{w}. For
details we refer to $\cite{liu},\cite{w}$.

Let $\mathbb{R}^{n+3}_1$ be the Lorentz space, i.e., $\mathbb{R}^{n+3}$ with inner
 product $<\cdot,\cdot>$ defined by
\[<x,y>=-x_0y_0+x_1y_1+\cdots+x_{n+2}y_{n+2},\] for
$x=(x_0,x_1,\cdots,x_{n+2}), y=(y_0,y_1,\cdots,y_{n+2})\in \mathbb{R}^{n+3}$.

Let $f:M^{n}\rightarrow \mathbb{R}^{n+1}$ be a hypersurface without umbilical points and
assume that $\{e_i\}$ is an orthonormal basis with respect to the
induced metric $I=df\cdot df$ with $\{\theta_i\}$ the dual basis.
Let $II=\sum_{ij}h_{ij}\theta_i\theta_j$ and
$H=\sum_i\frac{h_{ii}}{n}$ be the second fundamental form and the
mean curvature of $f$, respectively. We define the M\"{o}bius
position vector $Y: M^n\rightarrow \mathbb{R}^{n+3}_1$ of $f$ by
$$Y=\rho\left(\frac{1+|f|^2}{2},\frac{1-|f|^2}{2},f\right)~,
~~\rho^2=\frac{n}{n-1}(|II|^2-nH^2).$$
\begin{theorem}\cite{w}
Two hypersurfaces $f,\bar{f}: M^n\rightarrow \mathbb{R}^{n+1}$ are M\"{o}bius
equivalent if and only if there exists $T$ in the Lorentz group
$O(n+2,1)$ such that $\bar{Y}=YT.$
\end{theorem}
It follows immediately from Theorem 2.1 that
$$g=<dY,dY>=\rho^2df\cdot df$$
is a M\"{o}bius invariant, called the M\"{o}bius metric of $f$.

Let $\Delta$ be the Laplacian with respect to $g$. Define
$$N=-\frac{1}{n}\Delta Y-\frac{1}{2n^2}<\Delta Y,\Delta Y>Y,$$
which satisfies
$<Y,Y>=0=<N,N>, ~~<N,Y>=1.$

Let $\{E_1,\cdots,E_n\}$ be a local orthonormal basis for $(M^n,g)$
with dual basis $\{\omega_1,\cdots,\omega_n\}$. Write
$Y_i=E_i(Y)$. Then we have
$$<Y_i,Y>=<Y_i,N>=0, ~<Y_i,Y_j>=\delta_{ij}, ~~1\leq i,j\leq n.$$
Let $\xi$ be the mean curvature sphere of $f$ written as
$$\xi=\left(\frac{1+|f|^2}{2}H+f\cdot e_{n+1},\frac{1-|f|^2}{2}H-f\cdot e_{n+1},Hf+e_{n+1}\right)~,$$
where $e_{n+1}$ is the unit normal vector field of $f$ in $\mathbb{R}^{n+1}$.
Thus $\{Y,N,Y_1,\cdots,Y_n,\xi\}$ forms a moving frame in
$\mathbb{R}^{n+3}_1$ along $M^n$. We will use the following range of indices
in this section: $1\leq i,j,k\leq n$. We can write the structure
equations as following:
\begin{eqnarray*}
&&dY=\sum_iY_i\omega_i,\\
&&dN=\sum_{ij}A_{ij}\omega_iY_j+\sum_iC_i\omega_i\xi,\\
&&dY_i=-\sum_jA_{ij}\omega_jY-\omega_iN+\sum_j\omega_{ij}Y_j+\sum_jB_{ij}\omega_j\xi,\\
&&d\xi=-\sum_iC_i\omega_iY-\sum_{ij}\omega_iB_{ij}Y_j,
\end{eqnarray*}
where $\omega_{ij}$ is the connection form of the M\"{o}bius metric
$g$ and $\omega_{ij}+\omega_{ji}=0$. The tensors
$$
{\bf A}=\sum_{ij}A_{ij}\omega_i\otimes\omega_j,~~
{\bf B}=\sum_{ij}B_{ij}\omega_i\otimes\omega_j,~~
{\bf C}=\sum_iC_i\omega_i$$ are called the
Blaschke tensor, the M\"{o}bius second
fundamental form and the M\"{o}bius form of $f$, respectively.
The covariant derivative of
$C_i, A_{ij}, B_{ij}$ are defined by
\begin{eqnarray*}
&&\sum_jC_{i,j}\omega_j=dC_i+\sum_jC_j\omega_{ji},\\
&&\sum_kA_{ij,k}\omega_k=dA_{ij}+\sum_kA_{ik}\omega_{kj}+\sum_kA_{kj}\omega_{ki},\\
&&\sum_kB_{ij,k}\omega_k=dB_{ij}+\sum_kB_{ik}\omega_{kj}+\sum_kB_{kj}\omega_{ki}.
\end{eqnarray*}
The integrability conditions for the structure equations are given
by
\begin{eqnarray}
&&A_{ij,k}-A_{ik,j}=B_{ik}C_j-B_{ij}C_k,\label{equa1}\\
&&C_{i,j}-C_{j,i}=\sum_k(B_{ik}A_{kj}-B_{jk}A_{ki}),\label{equa2}\\
&&B_{ij,k}-B_{ik,j}=\delta_{ij}C_k-\delta_{ik}C_j,\label{equa3}\\
&&R_{ijkl}=B_{ik}B_{jl}-B_{il}B_{jk}
+\delta_{ik}A_{jl}+\delta_{jl}A_{ik}
-\delta_{il}A_{jk}-\delta_{jk}A_{il},\label{equa4}\\
&&R_{ij}:=\sum_kR_{ikjk}=-\sum_kB_{ik}B_{kj}+(tr{\bf
A})\delta_{ij}+(n-2)A_{ij},\label{equa5}\\
&&\sum_iB_{ii}=0, ~~\sum_{ij}(B_{ij})^2=\frac{n-1}{n}, ~~tr{\bf
A}=\sum_iA_{ii}=\frac{1}{2n}(1+\frac{n}{n-1}R),\label{equa6}
\end{eqnarray}
Here $R_{ijkl}$ denotes the curvature tensor of $g$,
and $R=\sum_{ij}R_{ijij}$ is the
M\"{o}bius scalar curvature. We know that all coefficients in the
structure equations are determined by $\{g, {\bf B}\}$ when $n\geq 3$. Thus we have
\begin{theorem}$\cite{w}$\label{fundthe}
Two hypersurfaces $f: M^n\rightarrow \mathbb{R}^{n+1}$ and $\bar{f}:
M^n\rightarrow \mathbb{R}^{n+1} (n\geq 3)$ are M\"{o}bius equivalent if and
only if there exists a diffeomorphism $\varphi: M^n\rightarrow M^n$
which preserves the M\"{o}bius metric and the M\"{o}bius second
fundamental form.
\end{theorem}
By equation (\ref{equa2}), we have
\begin{equation}\label{cb}
dC=0\Leftrightarrow\sum_k(B_{ik}A_{kj}-B_{jk}A_{ki})=0.
\end{equation}
For the second covariant derivative of $B_{ij}$ defined by
$$dB_{ij,k}+\sum_mB_{mj,k}\omega_{mi}+\sum_mB_{im,k}\omega_{mj}+
\sum_mB_{ij,m}\omega_{mk}=\sum_mB_{ij,km}\omega_m,$$
we have the following Ricci identities
$$B_{ij,kl}-B_{ij,lk}=\sum_mB_{mj}R_{mikl}+\sum_mB_{im}R_{mjkl}.$$
We call eigenvalues
of $(B_{ij})$ as \emph{M\"{o}bius principal curvatures} of $f$. Clearly the
number of distinct M\"{o}bius principal curvatures is the same as
that of its distinct Euclidean principal curvatures.

Let $k_1,\cdots,k_n$ be the principal curvatures of $f$, and
$\{\lambda_1,\cdots,\lambda_n\}$ the corresponding M\"{o}bius
principal curvatures, then the curvature sphere of principal
curvature $k_i$ is
$$\xi_i=\lambda_iY+\xi=\left(\frac{1+|f|^2}{2}k_i+f\cdot e_{n+1},\frac{1-|f|^2}{2}k_i-f\cdot e_{n+1},k_if+e_{n+1}\right)~.$$
Note that $k_i=0$ if and only if,
$$<\xi_i,(1,-1,0,\cdots,0)>=0.$$
This means that the curvature sphere of principal curvature $k_i$
is a hyperplane in $\mathbb{R}^{n+1}$.

\vskip 1 cm
\section{Generic conformally flat hypersurfaces in $\mathbb{R}^4$}
In this section, we give some local properties of the M\"{o}bius invariants of generic conformally flat hypersurfaces in $\mathbb{R}^4$.

Let $(M^n,g)$ be an n-dimensional Riemannian manifold,  and $\{e_1,\cdots e_n\}$ be a local orthonormal frame field on $(M^n,g)$, and $\{\omega _1,
\cdots ,\omega _n\}$ its dual coframe field. The Weyl conformal tensor $W=\sum_{ijkl}W_{ijkl}\omega_i\otimes\omega_j\otimes\omega_k\otimes\omega_l$
and the Schouten tensor $S=\sum_{ij}S_{ij}\omega_i\otimes\omega_j$ of $(M^n,g)$ are defined by, respectively,
\begin{equation*}
\begin{split}
&W_{ijkl}=R_{ijkl}-\frac{1}{n-2}\{R_{ik}\delta _{jl}-R_{jk}\delta _{il}+\delta _{ik}R_{jl}-
\delta _{jk}R_{il}-\frac{R}{(n-1)}(\delta_{ik} \delta_{jl} -\delta_{jk} \delta_{il})\},\\
&S_{ij}=R_{ij}-\frac{R}{2(n-1)}\delta_{ij},
\end{split}
\end{equation*}
where $R_{ij}$ denotes the Ricci curvature and $R$ the scalar curvature of $(M^n,g)$.

A result of Weyl  states that a Riemannian manifold $(M^n,g)$ of dimension $n (\geq 4)$ is conformally flat if and only if the Weyl conformal tensor vanishes,
and a Riemannian manifold $(M^n,g)$ of dimension $3$ is conformally flat if and only if the Schouten tensor is a Codazzi tensor (i.e., $S_{ij,k}=S_{ik,j}$).
Using the Weyl's result, we can prove the following lemma (or see \cite{yau}),
\begin{lemma}\cite{yau}\label{lem31}
A Riemannian product $(M_1,g_1)\times (M_2,g_2)=(M_1\times M_2, g_1+g_2)$ is conformally flat if and only if either

(1) $(M_i,g_i)$ is one dimensional curve, and $(M_j,g_j), (i\neq j)$ is a space form, or

(2) $(M_1,g_1)$ and $(M_2,g_2)$ are space forms of dimension at least two, with non-zero opposite curvatures.
\end{lemma}
For hypersurfaces in $\mathbb{R}^{n+1}$,
when $n\geq4$, it is well-known from the Cartan-Schouten that a hypersurface $f: M^n\to \mathbb{R}^{n+1}$ is conformally flat if and only if at least $n-1$ of the principal curvatures coincide
at each point. But Cartan-Schouten's result is no longer true in dimension
$3$, since there exist generic conformally flat hypersurfaces.

Let $f: M^3\to \mathbb{R}^4$ be a generic hypersurface. We choose an orthonormal basis $\{E_{1},E_2,E_{3}\}$ with respect to the M\"{o}bius metric $g$ such that
\begin{equation}\label{second}
(B_{ij})=diag\{b_1, b_2, b_3\},~~~b_1< b_2< b_3.
\end{equation}
Let $\{\omega_1,\omega_2,\omega_3\}$ be the dual of $\{E_{1},E_2,E_{3}\}$. The conformal fundamental forms of $f$ are defined by
$$\Theta_1=\sqrt{(b_3-b_1)(b_2-b_1)}\omega_1,~\Theta_2=\sqrt{(b_3-b_2)(b_2-b_1)}\omega_2,~\Theta_3=\sqrt{(b_3-b_1)(b_3-b_2)}\omega_3.$$

Using the equation (\ref{equa5}) and (\ref{equa6}), the  Schouten tensor of $f$ is
$$S=\sum_{ij}(-\sum_lB_{il}B_{lj}+A_{ij}+\frac{1}{6}\delta_{ij})\omega_i\wedge\omega_j.$$
Thus
\begin{equation}\label{schouten}
S_{ij,k}=-\sum_l(B_{il,k}B_{lj}+B_{il}B_{lj,k})+A_{ij,k}.
\end{equation}
If the hypersurface $f$ is conformally flat, then $S_{ij,k}=S_{ik,j}$. Combining the equations (\ref{equa1}) and (\ref{second}), we
obtain the following equation
\begin{equation}\label{flat1}
b_kB_{ik,j}-b_jB_{ij,k}=2(B_{ij}C_k-B_{ik}C_j).
\end{equation}
Using the equation (\ref{equa3}), we have the following equations,
\begin{equation}\label{flat2}
\begin{split}
&B_{12,3}=B_{13,2}=0,\\
&B_{ij,i}=\frac{3b_i}{b_j-b_i}C_j,~~B_{ii,j}=\frac{b_i-b_k}{b_j-b_i}C_j, ~~i\neq j, j\neq k, i\neq k.
\end{split}
\end{equation}
Using
$dB_{ij}+\sum_kB_{kj}\omega_{ki}+\sum_kB_{ik}\omega_{kj}=\sum_kB_{ij,k}\omega_k$
and (\ref{second}), we get
\begin{equation}\label{ww}
\omega_{ij}=\sum_k\frac{B_{ij,k}}{b_i-b_j}\omega_k=\frac{B_{ij,i}}{b_i-b_j}\omega_i+\frac{B_{ij,j}}{b_i-b_j}\omega_j.
\end{equation}
The following lemma is trivial by the equation (\ref{flat2}) and (\ref{ww}), (or see \cite{je3},\cite{su3}).
\begin{lemma}
Let $M^3\to \mathbb{R}^4$ be a generic hypersurface. The following are equivalent:
(1), the hypersurface is conformally flat;\\
(2), the schouten tensor is a Codazzi tensor;\\
(3), the conformal fundamental forms $\Theta_1,\Theta_2,\Theta_3$ are closed.
\end{lemma}
Next, we give the standard examples of generic conformally flat hypersurfaces in $\mathbb{R}^4$.
\begin{ex}\label{ex31}
Let $u: M^2\longrightarrow \mathbb{R}^3$ be an immersed surface. We define
\emph{the cylinder over $u$} in $\mathbb{R}^4$ as
\begin{equation*}
f=(id,u):\mathbb{R}^1\times M^2\longrightarrow
\mathbb{R}^1 \times \mathbb{R}^3=\mathbb{R}^4,~~~~
f(t,y)=(t,u(x)),
\end{equation*}
where $id:\mathbb{R}^1\longrightarrow \mathbb{R}^1$ is the identity map.
\end{ex}
 The first
fundamental form $I$ and the second fundamental form $II$ of
the cylinder $f$ are, respectively,
\begin{equation*}
I=I_{\mathbb{R}^1}+I_u, \;\; II=II_u,
\end{equation*}
where $I_u,II_u$ are the first and second fundamental forms of $u$,
respectively, and $I_{\mathbb{R}^1}$ denotes the standard metric of $\mathbb{R}^1$.
Let $\{k_1,k_2\}$ be principal curvatures of surface $u$. Obviously the principal curvatures of hypersurface $f$ are
$\{0,k_1,k_2\}.$
The M\"{o}bius metric $g$ of hypersurface $f$ is
\begin{equation}\label{eq-g1}
g=\rho^2I=\frac{n}{n-1}(|II|^2-nH^2)I
=\left(4H_u^2-3K_u\right)(I_{\mathbb{R}^1}+I_u),
\end{equation}
where $H_u,K_u$ are the mean curvature of $u$ and
Gaussian curvature of $u$, respectively. Therefore combining Lemma \ref{lem31} we have the following result.
\begin{PROPOSITION}\label{prop31}
Let $f:M^3\to \mathbb{R}^4$ be a cylinder over a surface $u: M^2\to \mathbb{R}^3$, then the cylinder $f$ is conformally flat if and only if
the surface $u$ is of  constant Gaussian curvature.
\end{PROPOSITION}
\begin{ex}\label{ex32}
Let $u:M^2\longrightarrow \mathbb{S}^3\subset \mathbb{R}^4$ be an
immersed surface.
We define \emph{the cone over $u$} in $\mathbb{R}^4$ as
\begin{equation*}
f:\mathbb{R}^+\times M^2\longrightarrow \mathbb{R}^4,
~~~~f(t,x)=tu(x).
\end{equation*}
\end{ex}
The first and second
fundamental forms of the cone $f$ are, respectively,
\[
I=I_{\mathbb{R}^1}+t^2I_u, \;\; II=t~II_u,
\]
where $I_u,II_u,I_{\mathbb{R}^1}$ are understood as before.
Let $\{k_1,k_2\}$ be principal curvatures of surface $u$.
The principal curvatures of hypersurface $f$ are
$\{0,\frac{1}{t}k_1,\frac{1}{t}k_2\}.$
The M\"{o}bius metric $g$ of hypersurface $f$ is
\begin{equation}\label{eq-g2}
g=\frac{1}{t^2}\left[4H_u^2-3(K_u-1)\right]
(I_{\mathbb{R}^1}+t^2I_u)
=\left[4H_u^2-3(K_u-1)\right]
(I_{\mathbb{H}^1}+I_u),
\end{equation}
where $H_u,K_u$ are the mean curvature and Gaussian curvature of
$u$, respectively, $I_{\mathbb{H}^1}=\frac{dt^2}{t^2}$. Therefore combining Lemma \ref{lem31} we have the following result.
\begin{PROPOSITION}\label{prop32}
Let $f:M^3\to \mathbb{R}^4$ be a cone over a surface $u: M^2\to \mathbb{S}^3$, then the cone $f$ is conformally flat if and only if
the surface $u$ is of constant Gaussian curvature.
\end{PROPOSITION}
\begin{ex}\label{ex33}
Let $\mathbb{R}^3_+=\{(x_1,x_2,x_3)\in \mathbb{R}^3|x_3>0\}$
be the upper half-space endowed with the standard
hyperbolic metric
\[
ds^2=\frac{1}{x_3^2}[dx_1^2+dx_2^2+dx_3^2]~.
\]
Let $u=(x_1,x_2,x_3):M^2\longrightarrow \mathbb{R}^3_+$ be
an immersed surface. We define
\emph{rotational hypersurface over $u$} in $\mathbb{R}^4$ as
\begin{equation*}
f:\mathbb{S}^1\times M^2\longrightarrow \mathbb{R}^4,~~~
f(\phi,x_1,x_2,x_3)=(x_1,x_2,x_3\phi),
\end{equation*}
where $\phi:\mathbb{S}^1\longrightarrow \mathbb{S}^1$ is the unit circle.
\end{ex}
The first fundamental form and the second fundamental form of $u$ is, respectively,
\begin{equation*}
\begin{split}
&I_u=\frac{1}{x_3^2}(dx_1\cdot dx_1+dx_2\cdot dx_2+dx_3\cdot dx_3),\\
&II_u=\frac{1}{x_3^2}(dx_1\cdot d\eta_1+dx_2\cdot d\eta_2
+dx_3\cdot d\eta_3)-\frac{\eta_3}{x_3}I_u.
\end{split}
\end{equation*}
The first and the second fundamental forms
of $f$ is,  respectively,
\[
I=dx\cdot dx=x_3^2(I_{\mathbb{S}^1}+I_u),
~~II=x_3II_u-\eta_3I_u-\eta_3I_{\mathbb{S}^1}.
\]
 Let
$\{k_1,k_2\}$ be principal curvatures of $u$. Then principal
curvatures of hypersurface $f$ are
$\{\frac{-\eta_3}{x_3^2},\frac{k_1}{x_3}-\frac{\eta_3}{x_3^2}~,~
\frac{k_2}{x_3}-\frac{\eta_3}{x_3^2}\}.$
Thus the M\"{o}bius metric of the rotational hypersurface $f$ is
\begin{equation}\label{eq-g3}
g=\rho^2I=\left[4H_u^2-3(K_u+1)\right]
(I_{\mathbb{S}^1}+I_u),
\end{equation}
where $H_u,K_u$ are the mean curvature and Gaussian curvature of
$u$, respectively.
Therefore combining Lemma \ref{lem31} we have the following result,
\begin{PROPOSITION}\label{prop33}
Let $f:M^3\to \mathbb{R}^4$ be a rotational hypersurface over a surface $u: M^2\to \mathbb{R}^3_+$, then the hypersurface $f$ is conformally flat if and only if
 the surface $u$ is of constant  Gaussian curvature.
\end{PROPOSITION}
\begin{PROPOSITION}\label{prop}
Let $f: M^3\to \mathbb{R}^4$ be one of  generic conformally flat hypersurfaces given by above three examples (\ref{ex31}) (\ref{ex32}) (\ref{ex33}).
Then the M\"{o}bius form is closed  $(i.e., dC=0)$.
\end{PROPOSITION}
\begin{proof}
The M\"{o}bius metric $g$ in above three examples (\ref{ex31}) (\ref{ex32}) (\ref{ex33}) can be unified in a single formula:
$$g=\left[4H_u^2-3(K_u-\epsilon)\right](ds^2+I_u)=\rho^2(ds^2+I_u),$$
where $I_u$, $K_u$, and $H_u$ are the induced metric, Gaussian curvature, and mean curvature of $u:M^2\rightarrow N^3(\epsilon)$, respectively.

Let $\{e_2,e_3\}$ be a local orthonormal basis on $TM^2$ with respect to $I_u$, consisting of unit principal vectors of $u$ and $e_1=\frac{\partial}{\partial s}$.
Then $\{e_1,e_2,e_3\}$ is an orthonormal basis for $T(I\times M^2)$ with respect to $ds^2+I_u$. Let $\tilde{R}_{ijkl}$ denote the curvature tensor of the metric $ds^2+I_u$, and  $R_{ijkl}$  the curvature tensor for $g=\rho^2[ds^2+I_u]$. From Yau's paper \cite{yau}, we have
\begin{equation}\label{rrij}
\begin{split}
&R_{ijij}=\rho^2\tilde{R}_{ijij}+\rho\rho_{ii}+\rho\rho_{jj}-|\nabla\rho|^2, i\neq j\\
&R_{ijik}=\rho^2\tilde{R}_{ijik}+\rho\rho_{jk},~~ when~ \{i,j,k\}~ are~ distinct,\\
\end{split}
\end{equation}
which implies that $(B_{ij})=diag(b_1,b_2,b_3)$ and $(A_{ij})=diag(a_1,a_2,a_3)$ under
the local orthonormal basis $\{\rho^{-1}e_1,\rho^{-1}e_2,\rho^{-1}e_3\}$ by the equation (\ref{equa4}).
Thus
$dC=0$ by the equation (\ref{cb}).
\end{proof}
Next, we prove Theorem \ref{the1}. From Proposition \ref{prop}, we prove another hand of Theorem \ref{the1} and we assume $dC=0$.
From the equation (\ref{cb}),
under the orthonormal basis $\{E_{1},E_2,E_{3}\}$ in (\ref{second}) we find
\begin{equation}\label{second1}
(A_{ij})=diag\{a_1, a_2, a_3\}.
\end{equation}
The equations (\ref{second}) and (\ref{second1}) imply that
$$R_{ijik}=A_{jk}=0,~~j\neq k,$$
by the equation (\ref{equa4}).

From the definition of $B_{ij,kl}$
and (\ref{equa3}), (\ref{flat2}) and (\ref{ww}), we have
\begin{equation}\label{bb1b}
\begin{split}
&B_{23,31}=\frac{3b_2B_{33,1}-3b_3B_{22,1}}{(b_3-b_2)^2}C_2+\frac{3b_3}{b_2-b_3}[C_{2,1}-\frac{B_{12,1}}{b_1-b_2}C_1]+B_{31,3}\frac{B_{12,1}}{b_1-b_2},\\
&B_{23,13}=(B_{22,1}-B_{33,1})\frac{B_{23,3}}{b_2-b_3}+(B_{11,2}-B_{33,2})\frac{B_{13,3}}{b_1-b_3},
\end{split}
\end{equation}
Using Ricci identity $B_{23,31}-B_{23,13}=(b_3-b_2)R_{2313}=0$ and (\ref{bb1b}),  we have
\begin{equation*}
b_3C_{1,2}=\frac{2b_1^2+b_2b_3}{(b_2-b_1)(b_1-b_3)}C_1C_2=-C_1C_2.
\end{equation*}
Similarly we have
\begin{equation*}
b_1C_{2,3}=-C_2C_3,~~ b_2C_{1,3}=-C_1C_3.
\end{equation*}
Therefore
\begin{equation}\label{bc}
b_kC_{i,j}=-C_iC_j,i\neq j,i\neq k,k\neq j.
\end{equation}
Now we define $\{C_{i,jk}\}$ given by
$$dC_{i,j}+\sum_mC_{m,j}\omega_{mi}+\sum_mC_{i,m}\omega_{mj}=\sum_mC_{i,jm}\omega_m.$$
Let $\{i,j,k\}$ be distinct. Taking derivative for (\ref{bc}) along $E_k$ and invoking (\ref{flat2}) and (\ref{ww}), we get
\begin{equation}\label{bc0}
\begin{split}
&B_{kk,k}C_{i,j}+b_k[C_{i,jk}-C_{k,j}\frac{B_{ki,k}}{b_k-b_i}-C_{i,k}\frac{B_{kj,k}}{b_k-b_j}]\\
&=-C_i[C_{j,k}-C_k\frac{B_{jk,k}}{b_k-b_j}]-C_j[C_{i,k}-C_k\frac{B_{ik,k}}{b_k-b_i}].
\end{split}
\end{equation}
If $b_1b_2b_3=0$, we can assume that $b_1=0$, which implies that $b_2=-b_3=\sqrt{\frac{1}{3}}$ by the equation (\ref{equa6}). Using
(\ref{flat2}), we have $C_1=C_2=C_3=0$ and $B_{ij,k}=0$.

Next we assume that $b_1b_2b_3\neq0$. Because
$b_1^2+b_2^2+b_3^2=\frac{2}{3},B_{ij,j}=B_{jj,i}-C_i$, from
(\ref{bc}) and (\ref{bc0}) we have
\begin{equation}\label{bcc}
b_kC_{i,jk}=-\frac{4}{3}\frac{C_iC_jC_k}{b_ib_jb_k}=-\frac{4}{3}\frac{C_1C_2C_3}{b_1b_2b_3}.
\end{equation}
Since $C_{i,jk}=C_{j,ik}=C_{k,ij}$ and $b_i\neq b_j,i\neq j$, from
(\ref{bcc}) we get
$$C_{i,jk}=C_{j,ik}=C_{k,ij}=0,~~C_1C_2C_3=0.$$
We can assume that $C_1=0$, then
\begin{equation}\label{ww1}
\omega_{12}=\frac{B_{12,1}}{b_1-b_2}\omega_1,~\omega_{13}=\frac{B_{13,1}}{b_1-b_3}\omega_1,
~\omega_{23}=\frac{B_{23,2}}{b_2-b_3}\omega_2+\frac{B_{23,3}}{b_2-b_3}\omega_3.
\end{equation}
From (\ref{ww1}),
combining $d\omega_{12}-\omega_{13}\wedge\omega_{32}=\frac{-1}{2}\sum_{kl}R_{12kl}\omega_k\wedge\omega_l,$  we obtain
\begin{equation*}
\begin{split}
&\frac{-1}{2}\sum_{kl}R_{12kl}\omega_k\wedge\omega_l=d(\frac{B_{12,1}}{b_1-b_2})\wedge\omega_1\\
&+[(\frac{B_{12,1}}{b_1-b_2})^2+\frac{B_{13,1}B_{23,2}}{(b_1-b_3)(b_2-b_3)}]\omega_1\wedge\omega_2+
\frac{B_{13,1}}{b_1-b_3}[\frac{B_{12,1}}{b_1-b_2}+\frac{B_{23,3}}{b_2-b_3}]\omega_1\wedge\omega_3,
\end{split}
\end{equation*}
which implies that
\begin{equation}\label{dc1}
\begin{split}
&E_3(\frac{B_{12,1}}{b_1-b_2})-\frac{B_{13,1}}{b_1-b_3}[\frac{B_{12,1}}{b_1-b_2}+\frac{B_{23,3}}{b_2-b_3}]=0,\\
&E_2(\frac{B_{12,1}}{b_1-b_2})-[(\frac{B_{12,1}}{b_1-b_2})^2+\frac{B_{13,1}B_{23,2}}{(b_1-b_3)(b_2-b_3)}]=R_{1212}=b_1b_2+a_1+a_2.
\end{split}
\end{equation}
Similarly we have
\begin{equation}\label{dc2}
\begin{split}
&E_2(\frac{B_{13,1}}{b_1-b_3})-\frac{B_{12,1}}{b_1-b_2}[\frac{B_{13,1}}{b_1-b_3}-\frac{B_{23,2}}{b_2-b_3}]=0,\\
&E_3(\frac{B_{13,1}}{b_1-b_3})-[(\frac{B_{13,1}}{b_1-b_3})^2-\frac{B_{12,1}B_{23,3}}{(b_1-b_2)(b_2-b_3)}]=R_{1313}=b_1b_3+a_1+a_3.
\end{split}
\end{equation}

Under the local basis above, $\{Y,N,Y_1,Y_2,Y_3,\xi\}$ forms a moving
frame in $\mathbb{R}^6_1$ along $M^3$. We define
\begin{equation}
\begin{split}
&F=b_1Y+\xi,\;\; X_2=\frac{B_{12,1}}{b_1-b_2}Y+Y_2,\;\;X_3=\frac{B_{13,1}}{b_1-b_3}Y+Y_3,\\
&T=a_1Y+N-\frac{B_{12,1}}{b_1-b_2}Y_2-\frac{B_{13,1}}{b_1-b_3}Y_3-b_1\xi,\\
&Q=2a_1+b_1^2+(\frac{B_{12,1}}{b_1-b_2})^2+(\frac{B_{13,1}}{b_1-b_3})^2.
\end{split}
\end{equation}
Clearly $F$ is the curvature sphere of principal curvature $k_1$. And
\begin{equation*}
\begin{split}
&<F,X_2>=<F,X_3>=<F,T>=<F,Y_1>=0,\\
&<T,X_2>=<T,X_3>=<T,Y_1>=<X_2,X_3>=0,\\
&<F,F>=<X_2,X_2>=<X_3,X_3>=1,~ <T,T>=Q.
\end{split}
\end{equation*}
From structure equation of the hypersurface and (\ref{ww1}), we get
\begin{equation}
E_1(F)=0,~~E_2(F)=(b_1-b_2)X_2, ~~ E_3(F)=(b_1-b_3)X_3.
\end{equation}
Thus curvature sphere $F$ induces a surface $\tilde{M}=M^3/L$ in the de-Sitter space $\mathbb{S}^5_1$
$$F:\tilde{M}=M^3/L\longrightarrow \mathbb{S}^5_1,$$
where fibers $L$ are integral submanifolds of distribution $D=span\{E_1\}$.
We define
$V=span\{T,Y_1\}.$ Clearly we have
$$F\bot V.$$
Using (\ref{ww1}), (\ref{dc1}) and (\ref{dc2}), we can get that
\begin{equation}\label{eqq}
\begin{split}
&E_1(Y_1)=-T,~~E_2(Y_1)=0,~~E_3(Y_1)=0,\\
&E_1(T)=QY_1,~~E_2(T)=\frac{B_{12,1}}{b_1-b_2}T,~~E_3(T)=\frac{B_{13,1}}{b_1-b_3}T.
\end{split}
\end{equation}
This implies that subspace $V$ is parallel along $M^3$.
Similarly we have
\begin{equation}\label{eq1}
E_1(Q)=0, E_2(Q)=2\frac{B_{12,1}}{b_1-b_2}Q, E_3(Q)=2\frac{B_{13,1}}{b_1-b_3}Q.
\end{equation}
Regarding (\ref{eq1}) as a linear first-order differential equation for
$Q$, we see that $Q\equiv 0$ or
$Q\neq 0$ on an connected manifold $M^n$. Therefore there are three possibilities for the induced metric on the fixed subspace $V\subset \mathbb{R}^6_1$.

{\bf Case 1}, $Q=0$, then $<T,T>=Q=0$, therefore $V$ is endowed with a degenerate inner product.

By (\ref{eqq}), $T$ determines a fixed light-like direction in $\mathbb{R}^6_1$. Up to a M\"{o}bius transformation, we may take to be
$$T=\lambda(1,-1,0,0,0,0), ~~\lambda\in C^{\infty}(M^3).$$
Since $V$ is a fixed degenerate subspace in $\mathbb{R}^6_1$, we can find a space-like vector
$v$ such that $V=Span\{e=(1,-1,0,0,0,0), v\}$ and $<e,F>=<v,F>=0$. We interpret the geometry of the hypersurface $f:M^3\rightarrow \mathbb{R}^4$
as below:

1) $v$ determines a fixed hyperplane $\Sigma$ in $\mathbb{R}^4$ because of $<T,v>=0$.

2) $F$ is a two parameter family of hyperplanes orthogonal to the fixed hyperplane $\Sigma$ in $\mathbb{R}^4$.

Therefore $f(M^3)$, as the envelope of this family of hyperplanes $F$, is clearly a cylinder over a surface $\tilde{M}\subset \mathbb{R}^3$.

{\bf Case 2}, $Q<0$, then $<T,T>$ is negative, and  $V$ is a Lorentz subspace in $\mathbb{R}^6_1$. Up to a M\"{o}bius transformation, we can assume that
\begin{equation*}
V=Span\{T,Y_1\}
=Span\{p_0=(1,1,0,0,0,0),p_1=(1,-1,0,0,0,0)\}.
\end{equation*}

Using the stereographic projection, $p_0, p_1$ correspond to the origin $O$ and the point at infinity $\infty$ of $\mathbb{R}^4$, respectively. Since $F\perp V$,
$F$ is a two parameter family of hyperplanes (passing $O$ and $\infty$). Therefore $f(M^3)$, as the envelope of this family of hyperplanes $F$, is clearly a cone (with
vertex $O$) over a surface $\tilde{M}\subset \mathbb{S}^3$.

{\bf Case 3}, $Q>0$, then $<T,T>$ is positive, and  $V$ is a space-like subspace in $\mathbb{R}^6_1$. Up to a M\"{o}bius transformation, we can assume that
\begin{equation*}
V=Span\{T,Y_1\}
=Span\{(0,0,1,0,0,0),(0,0,0,1,0,0)\}=\mathbb{R}^2.
\end{equation*}
Thus $V$ is a fixed two dimensional plane $\mathbb{R}^2\subset \mathbb{R}^4$, and $F$ is a two parameter family of hyper-sphere orthogonal to this fixed plane $\mathbb{R}^2$ with centers locating
on it. Thus $F$ envelopes a rotational hypersurface $f(M^3)$ (over a surface $\tilde{M}\subset \mathbb{R}^3_+$).

From Case 1, Case 2, Case 3, we prove that if $dC=0$, then the hypersurface is  M\"{o}bius equivalent to one of the standard examples of generic
conformally flat hypersurface. thus
we complete the proof of Theorem \ref{the1}.

\vskip 1 cm
\section{Global behavior of the generic conformally flat hypersurface}
Let $f: M^3\to \mathbb{R}^4$ be a generic conformally flat hypersurface. We say that the pair $(U,\omega)$ is admissible if\\
(1), $U$ is an open subset of $M^3$,\\
(2), $\omega=(\omega_1,\omega_2,\omega_3)$ is a orthonormal co-frame field on $U$ with respect to the M\"{o}bius metric $g$,\\
(3), $\omega_1\wedge\omega_2\wedge\omega_3=dv_g$,\\
(4), $B=\sum_ib_i\omega_i\otimes\omega_i$.

Denote by $F=(E_1,E_2,E_3)$ the dual frame field of $\omega$. Then it is easily-seen
that, $(U,\omega)$ is admissible if and only if $E_i$ is an unit principal vector associated to
$b_i$ for each $1\leq i\leq 3$, and $\{E_1,E_2,E_3\}$ is an oriented basis associated to the
orientation of $M^3$. Denote by $\{\omega_{ij}\}$ the connection form with respect
to $(U,\omega_1,\omega_2,\omega_3)$.
Thus under the admissible frame field $\{E_1,E_2,E_3\}$,
\begin{equation*}
(B_{ij})=diag\{b_1, b_2, b_3\}.
\end{equation*}
Now we introduce two $2$-forms on $M^3$: for every admissible co-frame field $(U,\omega)$,
set
\begin{equation*}
\begin{split}
&\Phi=\omega_{12}\wedge\omega_3+\omega_{23}\wedge\omega_1+\omega_{31}\wedge\omega_2,\\
&\Psi=(b_1-b_2)^2\omega_{12}\wedge\omega_3+(b_2-b_3)^2\omega_{23}\wedge\omega_1+(b_1-b_3)^2\omega_{31}\wedge\omega_2.
\end{split}
\end{equation*}
If $(U,\omega)$ and $(\tilde{U},\tilde{\omega})$ are both admissible co-frame
fields with $U\cap \tilde{U}\neq\emptyset$. Then on $U\cap \tilde{U}$, $\omega_i=\epsilon_i\tilde{\omega}_i, ~~\omega_{ij}=\epsilon_i\epsilon_j\tilde{\omega}_{ij}$ for every
$1\leq i\leq 3$, where $\epsilon_i=1$ or $-1$ and $\epsilon_1\epsilon_2\epsilon_3=1$. Thus
$$\omega_{12}\wedge\omega_3=\tilde{\omega}_{12}\wedge\tilde{\omega}_3,~~
\omega_{23}\wedge\omega_1=\tilde{\omega}_{23}\wedge\tilde{\omega}_1,~~
\omega_{31}\wedge\omega_2=\tilde{\omega}_{31}\wedge\tilde{\omega}_2.$$
Therefore the $2$-forms $\Phi, \Psi$ are well-defined on $M^3$.
Combining $d\omega_{ij}-\sum_k\omega_{ik}\wedge\omega_{kj}=\frac{-1}{2}\sum_{kl}R_{ijkl}\omega_k\wedge\omega_l$, $d\omega_i=\sum_k\omega_{ik}\wedge\omega_k$
and the equation (\ref{flat2}), (\ref{ww}), we get
\begin{equation*}
\begin{split}
&d(\omega_{12}\wedge\omega_3)=-R_{1212}\omega_1\wedge\omega_2\wedge\omega_3\\
&+[\frac{-9b_1b_2C_3^2}{(b_1-b_3)^2(b_2-b_3)^2}+\frac{9b_2b_3C_1^2}{(b_1-b_2)^2(b_1-b_3)^2}
+\frac{9b_1b_3C_2^2}{(b_1-b_2)^2(b_2-b_3)^2}]\omega_1\wedge\omega_2\wedge\omega_3.
\end{split}
\end{equation*}
Similarly we can compute $d(\omega_{23}\wedge\omega_1)$ and $d(\omega_{31}\wedge\omega_2)$. Therefore we have
\begin{equation}\label{deq}
\begin{split}
&d\Phi=[\frac{9b_1b_2C_3^2}{(b_1-b_3)^2(b_2-b_3)^2}-R_{1212}-R_{1313}-R_{2323}]\omega_1\wedge\omega_2\wedge\omega_3\\
&+[\frac{9b_2b_3C_1^2}{(b_1-b_2)^2(b_1-b_3)^2}
+\frac{9b_1b_3C_2^2}{(b_1-b_2)^2(b_2-b_3)^2}]\omega_1\wedge\omega_2\wedge\omega_3.
\end{split}
\end{equation}
Using $db_i=\sum_kB_{ii,k}\omega_k$ and the same computation as $d\Phi$, we can obtain
\begin{equation}\label{deq1}
\begin{split}
&d\Psi=-[(b_1-b_2)^2R_{1212}+(b_1-b_3)^2R_{1313}+(b_2-b_3)^2R_{2323}]dv_g\\
&+[\frac{18b_1b_2C_3^2}{(b_1-b_3)^2(b_2-b_3)^2}+\frac{18b_2b_3C_1^2}{(b_1-b_2)^2(b_1-b_3)^2}
+\frac{18b_1b_3C_2^2}{(b_1-b_2)^2(b_2-b_3)^2}]dv_g,
\end{split}
\end{equation}
where $dv_g=\omega_1\wedge\omega_2\wedge\omega_3$. Combining the equation (\ref{deq}) and the equation (\ref{deq1}), we have
\begin{equation}\label{deq2}
2d\Phi-d\Psi=\{[(b_1-b_2)^2-2]R_{1212}+[(b_2-b_3)^2-2]R_{2323}+[(b_1-b_3)^2-2]R_{1313}\}dv_g.
\end{equation}
If $M^3$ is compact,  the equation (\ref{deq2}) implies that
\begin{equation}\label{deq3}
\int_{M^3}\{[(b_1-b_2)^2-2]R_{1212}+[(b_2-b_3)^2-2]R_{2323}+[(b_1-b_3)^2-2]R_{1313}\}dv_g=0.
\end{equation}

From $b_1+b_2+b_3=0$ and $b_1^2+b_2^2+b_3^2=\frac{2}{3}$, we can derive that
\begin{equation*}
(b_1-b_2)^2+(b_2-b_3)^2+(b_1-b_3)^2=2.
\end{equation*}
which implies that
\begin{equation}\label{bb2}
(b_1-b_2)^2-2<0,~~(b_2-b_3)^2-2<0,~~(b_1-b_3)^2-2<0.
\end{equation}

Now we assume that the sectional curvature of $M^3$ with respect to the M\"{o}bius metric have sign, for example, the sectional
curvature is nonnegative. The equation (\ref{deq3}) and (\ref{bb2}) imply that the sectional curvature vanishes, i.e.,
$$R_{1212}=R_{2323}=R_{1313}=0.$$
In \cite{lmw}, authors classify the hypersurfaces $f: M^3\to \mathbb{R}^{4}$ with constant M\"{o}bius sectional curvature, which are
non-compact. This is a contradiction, thus we finish the proof of Theorem \ref{the2}.

Using the equation (\ref{equa4}) and the equation (\ref{equa6}), we have
\begin{equation}\label{deq4}
\begin{split}
&[(b_1-b_2)^2-2]R_{1212}+[(b_2-b_3)^2-2]R_{2323}+[(b_1-b_3)^2-2]R_{1313}\\
&=\frac{2}{9}-\frac{10}{3}tr(A)+3[b_1^2a_1+b_2^2a_2+b_3^2a_3].
\end{split}
\end{equation}
On the other hand,  the equation (\ref{equa5}) implies that
\begin{equation}\label{bb3}
|Ric|^2=\frac{2}{9}+5tr(A)^2+|A|^2-\frac{4}{3}tr(A)-2[b_1^2a_1+b_2^2a_2+b_3^2a_3],
\end{equation}
where $|Ric|$ denote the norm of the Ricci curvature.
Combining the equation (\ref{deq4}) and the equation (\ref{bb3}), we can derive that
\begin{equation}\label{deq5}
\begin{split}
&[(b_1-b_2)^2-2]R_{1212}+[(b_2-b_3)^2-2]R_{2323}+[(b_1-b_3)^2-2]R_{1313}\\
&=\frac{5}{9}-\frac{16}{3}tr(A)+\frac{15}{2}tr(A)^2+\frac{3}{2}|A|^2-\frac{3}{2}|Ric|^2.
\end{split}
\end{equation}
Let $\tilde{A}:=A-\frac{1}{3}tr(A)g$ denote the trace-free Blaschke tensor, then $|\tilde{A}|^2=|A|^2-\frac{1}{3}tr(A)^2$.
Thus from the equation (\ref{deq5}), we have
\begin{equation}\label{deq6}
\begin{split}
&[(b_1-b_2)^2-2]R_{1212}+[(b_2-b_3)^2-2]R_{2323}+[(b_1-b_3)^2-2]R_{1313}\\
&=\frac{3}{2}|\tilde{A}|^2-\frac{3}{2}[|Ric|^2-\frac{1}{3}R^2]-\frac{1}{9}.
\end{split}
\end{equation}
Now if the hypersurface $M^3$ is compact, then
\begin{equation}\label{inte}
\int_{M^3}\{|\tilde{A}|^2+\frac{1}{3}R^2-|Ric|^2-\frac{2}{27}\}dv_g=0.
\end{equation}
Therefore we finish the proof of Theorem \ref{the3}.

Since $R=tr(Ric)$, we have $\frac{1}{3}R^2-|Ric|^2\leq 0$ on $M^3$. if $\frac{1}{3}R^2-|Ric|^2\equiv0$, then the sectional curvature
$K=0$, and there is a contradiction by the results in \cite{lmw}. Thus Corollary \ref{cor} is proved.

\end{document}